\documentclass[11pt]{article}
\usepackage{amscd}
\usepackage{amsfonts}
\usepackage{amsmath}
\usepackage{amssymb}
\usepackage{amsthm}
\usepackage{bbm}
\usepackage{CJK}
\usepackage{fancyhdr}
\usepackage{graphicx}
\usepackage{hyperref}
\usepackage{indentfirst}
\usepackage{latexsym}
\usepackage{mathrsfs}
\usepackage{xypic}
\usepackage[square, comma, sort&compress, numbers]{natbib}

\newtheorem{theorem}{Theorem}[section]
\newtheorem{lemma}[theorem]{Lemma}
\newtheorem{definition}[theorem]{Definition}
\newtheorem{proposition}[theorem]{Proposition}
\newtheorem{example}[theorem]{Example}

\newtheorem{corollary}[theorem]{Corollary}
\newtheorem{remark}[theorem]{Remark}

\usepackage[top=1in,bottom=1in,left=1.25in,right=1.25in]{geometry}
\def\<{\langle}
\def\>{\rangle}
\def\a{\alpha}
\def\b{\beta}

\def\g{\gamma}

\def\l{\lambda}
\def\lr{\longrightarrow}

\def\o{\otimes}

\def\r{\rho}

\def\t{\times}

\date{}
\begin{document}
\renewcommand{\baselinestretch}{1.2}
\renewcommand{\arraystretch}{1.0}
\title{\bf    Crossed product Hom-Hopf algebras and lazy 2-cocycle}
\author{{\bf Daowei Lu$^1$, Shuangjian Guo$^2$
}\\
{\small $^1$Department of Mathematics, Jining University}\\
{\small Qufu 273155, Shandong, P. R. China}\\
{\small $^2$School of Mathematics and Statistics, Guizhou University of Finance and Economics
}\\
{\small Guiyang 550025, Guizhou, P. R. China}
\\}
 \maketitle
\begin{center}
\begin{minipage}{12.cm}

\noindent{\bf Abstract.} Let $(H,\a)$ be a Hom-Hopf algebra and $(A,\b)$ be a Hom-algebra. In this paper we will construct the Hom-crossed product $(A\#_\sigma H,\b\o\a)$, and prove that the extension $A\subseteq A\#_\sigma H$ is actually a Hom-type cleft extension and vice versa. Then we will give the necessary and sufficient conditions to make $(A\#_\sigma H,\b\o\a)$ into a Hom-Hopf algebra. Finally we will study the lazy 2-cocycle on $(H,\a)$.
 \\

\\

 \end{minipage}
 \end{center}
 \normalsize\vskip1cm

\section*{Introduction}
The crossed products of algebras with Hopf algebras were independently introduced in \cite{BCM} and \cite{DT}. Blattner and Montgomery showed in \cite{BM} that a crossed product with invertible 2-cocycle is a cleft extension. In particular, crossed products provide examples of Hopf-Galois extensions. Conversely, a Hopf-Galois extension with normal basis property is a crossed product, see \cite{BM}. In the paper \cite{WJZ}, the authors gave the necessary and sufficient conditions for a crossed product to form a bialgebra, even a Hopf algebra, which is a more general structure than the Radford biproduct.

Algebraic deformation has been well developed recently, and its theory has been applied in modules of quantum phenomena, as well as in analysis of complex systems. Hom-type algebras appeared first in physical contexts, in connection with twisted, discretized or deformed derivatives and corresponding generalizations, discretizations and deformations of vector fields and differential calculus (see \cite{AGS,AS,C1,C2,C3,C4,H}). Hom-type algebras have been introduced in the form of Hom-Lie algebras in \cite{HLS}, where the Jacobi identity was twisted along a linear endomorphism. Meanwhile, Hom-associative algebras have been suggested in \cite{MS1} to give rise to a Hom-Lie algebra using the commutator bracket. Other Hom-type structures such as Hom-coalgebras, Hom-bialgebras, Hom-Hopf algebras and their properties have been considered in \cite{MS2,Y}. The authors \cite{MP2} introduced Hom-analogues of twisted tensor product and smash product by $twisting\  principle$, and in \cite{LW} the Hom-type smash coproduct and Majid bicrossproduct was constructed.

Let $H$ be a Hopf algebra. A left 2-cocycle $\sigma:H\o H\rightarrow k$ is called lazy if it satisfies the condition
$$\sigma(h_{1},g_{1})h_{2}g_{2}=h_{1}g_{1}\sigma(h_{2},g_{2}),$$
for all $h,g\in H$. This kind of cocycles were used in \cite{Ca} to compare the Brauer groups of Sweedler's Hopf algebra with respect to the different quasitriangular structures. Lazy 2-cocycle has been recently developed in \cite{BC} which mainly stated that the set $Z^{2}_{L}(H)$ of all normalized and convolution invertible lazy 2-cocycles on $H$ form a group and defined the second lazy cohomology group $H^{2}_{L}(H)=Z^{2}_{L}(H)/B^{2}_{L}(H)$, where $B^{2}_{L}(H)$ is lazy 2-coboundary. This group generalizes the second Sweedler cohomology
group of a cocommutative Hopf algebra. Moreover the group $H^{2}_{L}(H)$ could be imbedded as a subgroup into $Bigal(H)$, the group of Bigalois objects of $H$.

Motivated by these ideas, in this paper, firstly we will construct the Hom-crossed product, and prove the equivalence between crossed products and cleft extensions. Then we will give the necessary and sufficient conditions for a crossed product to form a Hom-Hopf algebra. Finally we will establish the lazy 2-cocycle in the setting of Hom-Hopf algebra.

This paper is organized as follows:

In section 1, we will recall the basic definitions and results on Hom-Hopf algebra, such as Hom-module, Hom-comodule, Hom-smash product and Hom-smash coproduct.

In section 2, we will construct the Hom-crossed product, and prove the equivalence between crossed products and cleft extensions (see Theorem 2.10).

In section 3, We will give the necessary and sufficient conditions for a crossed product to form a Hom-Hopf algebra (see Proposition 3.2 and 3.5).

In section 4, we will define the Hom-type lazy 2-cocycle and prove that all normalized and convolution invertible lazy 2-cocycles form a group $Z^{2}_{L}(H)$. Then we will extend this kind of 2-cocycle to Drinfeld double $D(H)$ and Radford biproduct. In the end we will use the lazy 2-cocycle to construct the duals of the objects in the more general Yetter-Drinfeld category (see Proposition 4.11).

Throughout this article, all the vector spaces, tensor product and homomorphisms are over a fixed field $k$ unless otherwise stated. We use the Sweedler's notation for the terminologies on coalgebras. For a coalgebra $C$, we write comultiplication $\Delta(c)=\sum c_{1}\otimes c_{2}$ for any $c\in C$.
\section{Preliminary}

In this section, we will recall the definitions in \cite{MP} on Hom-Hopf algebras,  Hom-modules and  Hom-comodules.

 A unital Hom-associative algebra is a triple $(A,\mu,\alpha)$ where $\alpha:A\lr A$ and $\mu:A\o A\lr A$ are linear maps, with notation $\mu(a\o b)=ab$ such that for any $a,b,c\in A$,
 \begin{eqnarray*}
&&\alpha(ab)=\alpha(a)\alpha(b),\ \alpha(1_{A})=1_{A},\\
&&1_{A}a=\alpha(a)=a1_{A},\ \alpha(a)(bc)=(ab)\alpha(c).
\end{eqnarray*}
A linear map $f:(A,\mu_{A},\alpha_{A})\lr (B,\mu_{B},\alpha_{B})$ is called a morphism of Hom-associative algebra if $\alpha_{B}\circ f=f\circ\alpha_{A}$, $f(1_{A})=1_{B}$ and $f\circ\mu_{A}=\mu_{B}\circ(f\o f).$

 A counital Hom-coassociative coalgebra is a triple $(C,\Delta,\varepsilon,\alpha)$ where $\alpha:C\lr C$, $\varepsilon:C\lr k$, and $\Delta:C\lr C\o C$ are linear maps such that
 \begin{eqnarray*}
&&\varepsilon\circ\alpha=\varepsilon,\ (\alpha\o\alpha)\circ\Delta=\Delta\circ\alpha,\\
&&(\varepsilon\o id)\circ\Delta=\alpha=(id\o\varepsilon)\circ\Delta,\\
&&(\Delta\o\alpha)\circ\Delta=(\alpha\o\Delta)\circ\Delta.
\end{eqnarray*}

A linear map $f:(C,\Delta_{C},\alpha_{C})\lr (D,\Delta_{D},\alpha_{D})$ is called a morphism of Hom-coassociative coalgebra if $\alpha_{D}\circ f=f\circ\alpha_{C}$, $\varepsilon_{D}\circ f=\varepsilon_{C}$ and $\Delta_{D}\circ f=(f\o f)\circ\Delta_{C}.$

In what follows, we will always assume all Hom-algebras are unital and Hom-coalgebras are counital.

A Hom-bialgebra is a quadruple $(H,\mu,\Delta,\alpha)$, where $(H,\mu,\alpha)$ is a Hom-associative algebra and $(H,\Delta,\alpha)$ is a Hom-coassociative coalgebra such that $\Delta$ and $\varepsilon$ are morphisms of Hom-associative algebra.

A Hom-Hopf algebra $(H,\mu,\Delta,\alpha)$ is a Hom-bialgebra $H$ with a linear map $S:H\lr H$(called antipode) such that
\begin{eqnarray*}
&&S\circ\alpha=\alpha\circ S,\\
&&S(h_{1})h_{2}=  h_{1}S(h_{2})=\varepsilon(h)1,
\end{eqnarray*}
for any $h\in H$. For $S$ we have the following properties:
 \begin{eqnarray*}
&&S(h)_{1}\o S(h)_{2}=  S(h_{2})\o S(h_{1}),\\
&&S(gh)=S(h)S(g),\ \varepsilon\circ S=\varepsilon.
 \end{eqnarray*}
For any Hopf algebra $H$ and any Hopf algebra endomorphism $\alpha$ of $H$, there exists a Hom-Hopf algebra $H_{\alpha}=(H,\alpha\circ\mu,1_{H},\Delta\circ\alpha,\varepsilon,S,\alpha)$.

Let $(A,\alpha_{A})$ be a Hom-associative algebra, $M$ a linear space and $\alpha_{M}:M\lr M$ a linear map. A left $A$-module structure on $(M,\alpha_{M})$ consists of a linear map $A\o M\lr M$, $a\o m\mapsto a\cdot m$, such that
 \begin{eqnarray*}
&&1_{A}\cdot m=\alpha_{M}(m),\\
&&\alpha_{M}(a\cdot m)=\alpha_{A}(a)\cdot\alpha_{M}(m),\\
&&\alpha_{A}(a)\cdot(b\cdot m)=(ab)\cdot\alpha_{M}(m),
\end{eqnarray*}
for any $a,b\in A$ and $m\in M.$

Similarly we can define the right $(A,\alpha)$-modules. Let $(M,\mu)$ and $(N,\nu)$ be two left $(A,\alpha)$-modules, then a linear map $f:M\lr N$ is a called left $A$-module map if $f(am)=af(m)$ for any $a\in A$, $m\in M$ and $f\circ\mu=\nu\circ f$.

Let $(C,\alpha_{C})$ be a Hom-coassociative coalgebra, $M$ a linear space and $\alpha_{M}:M\lr M$ a linear map. A right $C$-comodule structure on $(M,\alpha_{M})$ consists a linear map $\rho:M\lr M\o C$ such that
 \begin{eqnarray*}
&&(id\o\varepsilon_{C})\circ\rho=\alpha_{M},\\
&&(\alpha_{M}\o\alpha_{C})\circ\rho=\rho\circ\alpha_{M},\\
&&(\rho\o\alpha_{C})\circ\rho=(\alpha_{M}\o\Delta)\circ\rho.
 \end{eqnarray*}
Let $(M,\mu)$ and $(N,\nu)$ be two right $(C,\gamma)$-comodules, then a linear map $g:M\lr N$ is a called right $C$-comodule map if $g\circ \mu=\nu\circ g$ and $\rho_{N}\circ g=(g\otimes id)\circ\rho_{M}$.

Let $(H,\mu_{H},\Delta_{H},\alpha_{H})$ be a Hom-bialgebra. A Hom-associative algebra $(A,\mu_{A},\alpha_{A})$ is called a left $H$-module Hom-algebra if $(A,\alpha_{A})$ is a left $H$-module, with the action $H\o A\lr A,\ h\o a\mapsto h\cdot a$, such that
 \begin{eqnarray*}
&&\alpha^{2}_{H}(h)\cdot (ab)=  (h_{1}\cdot a)(h_{2}\cdot b),\\
&&h\cdot 1_{A}=\varepsilon(h)1_{A},
\end{eqnarray*}
for all $h\in H$ and $a,b\in A$.

When $A$ is a left $H$-module Hom-algebra, in \cite{MP} the Hom-smash product $A\#H$ is defined as follows:
$$(a\#h)(b\#k)=  a(\alpha^{-2}_{H}(h_{1})\cdot\alpha^{-1}_{A}(b))\#\alpha^{-1}_{H}(h_{2})k,$$
for all $a,b\in A$ and $h,k\in H$.

Recall from \cite{LW} that a Hom-coalgebra $(C,\g)$ is Hom-comodule coalgebra if it is a left Hom-comdule over the Hom-bialgebra $(H,\a)$ and satisfies the following relation:
\begin{eqnarray*}
&&\a^2(c_{(-1)})\o c_{(0)1}\o c_{(0)2}=c_{1(-1)}c_{2(-1)}\o c_{1(0)}\o c_{2(0)},\\
&&\varepsilon(c_{(-1)})c_{(0)}=\varepsilon(c)1.
\end{eqnarray*}
Then we have the Hom-smash coproduct $(C\times H,\g\o\a)$ with the comultiplication and counit
\begin{eqnarray*}
&&\Delta(c\times h)=c_1\t\g^{-2}(c_{2(-1)})\a^{-1}(h_1)\o\g^{-1}(c_{2(0)})\t h_2,\\
&&\varepsilon(c\t h)=\varepsilon_C(c)\varepsilon_H(h).
\end{eqnarray*}

\section{Hom-Crossed product}
\def\theequation{2.\arabic{equation}}
\setcounter{equation} {0} \hskip\parindent

In this section we will construct the crossed product on the Hom-Hopf algebra and prove that Hom-crossed product is equivalent to Hom-cleft extension.

\begin{definition}
Let $(H,\a)$ be a Hom-Hopf algebra and $(A,\b)$ a Hom-algebra. We say that $H$ weakly acts on $A$ from the left if there is a linear map, given by $\cdot:H\o A\rightarrow A$, such that for all $h\in H$ and $a,b\in A$
$$\b(h\cdot a)=\a(h)\cdot \b(a)$$
and
$$\a^2(h)\cdot(ab)=(h_1\cdot a)(h_2\cdot b),~~h\cdot 1=\varepsilon_H(h)1.$$
\end{definition}

\begin{proposition}
Let $(H,\a)$ be a Hom-Hopf algebra and $(A,\b)$ a Hom-algebra. Assume that $H$ weakly acts on $A$ from the left, then $(A\#_\sigma H,\b\o\a)$ is a Hom-algebra under the following multiplication
$$(a\#h)(b\#g)=a[(\a^{-4}(h_{11})\cdot \b^{-2}(b))\sigma(\a^{-3}(h_{12}),\a^{-2}(g_{1}))]\#\a^{-1}(h_2g_2),
$$
for all $a,b\in A$ and $h,g\in H$,
if and only if
\begin{itemize}
  \item [(1)] $A$ is a twisted Hom-$H$-module, that is, $1\cdot a=\b(a)$ for all $a\in A$, and
  $$(h_1\cdot(\a^{-1}(l_1)\cdot a))\sigma(\a(h_2),\a(l_2))=\sigma(\a(h_1),\a(l_1))(\a^{-1}(h_2l_2)\cdot \b(a)),$$
  for all $h,l\in H.$
  \item [(2)] $\sigma$ is normal, namely for all $h\in H$,
  $$\sigma(h,1)=\sigma(1,h)=\varepsilon_H(h),~~\sigma\circ(\a\o\a)=\b\circ\sigma.$$
  \item [(3)]For all $h,l,m\in H,$
  $$(h_1\cdot\sigma(l_1,m_1))\sigma(\a(h_2),l_2m_2)=\sigma(\a(h_1),\a(l_1)))\sigma(h_2l_2,\a^2(m)).$$
\end{itemize}
\end{proposition}

\begin{proof}
The direction $(\Rightarrow)$ is a routine exercise and we only prove the other direction. For all $a,b,c\in A$ and $h,g,l\in H$,
$$\begin{aligned}
&[(a\#h)(b\#g)](\b(c)\#\a(l))\\
=&\{a[(\a^{-4}(h_{11})\cdot \b^{-2}(b))\sigma(\a^{-3}(h_{12}),\a^{-2}(g_{1}))]\#\a^{-1}(h_2g_2)\}(\b(c)\#\a(l))\\
=&\{[\a^{-1}(a)(\a^{-4}(h_{11})\cdot\b^{-2}(b))][\sigma(\a^{-3}(h_{12}),\a^{-2}(g_1))(\a^{-6}(h_{211}g_{211})\cdot\b^{-2}(c))]\}\\
&\sigma(\a^{-3}(h_{212}g_{212}),l_1)\#\a^{-2}(h_{22}g_{22})l_2\\
\stackrel{(1)}{=}&\{[\a^{-1}(a)(\a^{-4}(h_{11})\cdot\b^{-2}(b))][(\a^{-6}(h_{1211})\cdot(\a^{-6}(g_{111})\cdot\b^{-3}(c)))\sigma(\a^{-5}(h_{1212}),\a^{-4}(g_{112}))]\}\\
&\sigma(\a^{-3}(h_{122})\a^{-2}(g_{12}),l_1)\#\a^{-1}(h_{2}g_{2})l_2\\
=&\{[\a^{-1}(a)((\a^{-5}(h_{11})\cdot\b^{-3}(b))(\a^{-5}(h_{12})\cdot(\a^{-7}(g_{111})\cdot\b^{-4}(c))))]\sigma(\a^{-3}(h_{211}),\a^{-3}(g_{112}))\}\\
&\sigma(\a^{-3}(h_{212})\a^{-2}(g_{12}),l_1)\#\a^{-1}(h_{2}g_{2})l_2\\
=&[a(\a^{-2}(h_{1})\cdot(\b^{-2}(b)(\a^{-6}(g_{111})\cdot\b^{-3}(c)))][\sigma(\a^{-3}(h_{211}),\a^{-3}(g_{112}))\\
&\sigma(\a^{-4}(h_{212})\a^{-3}(g_{12}),\a^{-1}(l_1))]\#\a^{-1}(h_{2}g_{2})l_2\\
=&[a(\a^{-2}(h_{1})\cdot(\b^{-2}(b)(\a^{-4}(g_{1})\cdot\b^{-3}(c)))][\sigma(\a^{-3}(h_{211}),\a^{-3}(g_{211}))\\
&\sigma(\a^{-4}(h_{212})\a^{-4}(g_{212}),\a^{-1}(l_1))]\#(\a^{-1}(h_{2})\a^{-2}(g_{22}))l_2\\
\stackrel{(3)}{=}&[a(\a^{-2}(h_{1})\cdot(\b^{-2}(b)(\a^{-4}(g_{1})\cdot\b^{-3}(c)))][(\a^{-4}(h_{211})\cdot\sigma(\a^{-4}(g_{211}),\a^{-3}(l_{11})))\\
&\sigma(\a^{-3}(h_{212}),\a^{-4}(g_{212})\a^{-3}(l_{12}))]\#(\a^{-1}(h_{2})\a^{-2}(g_{22}))l_2\\
=&\{a[(\a^{-3}(h_{1})\cdot(\b^{-3}(b)(\a^{-5}(g_{1})\cdot\b^{-4}(c))))(\a^{-5}(h_{211})\cdot\sigma(\a^{-5}(g_{211}),\a^{-4}(l_{11})))]\}\\
&\sigma(\a^{-2}(h_{212}),\a^{-3}(g_{212})\a^{-2}(l_{12}))\#(\a^{-1}(h_{2})\a^{-2}(g_{22}))l_2\\
=&\{a[\a^{-3}(h_{11})\cdot((\b^{-3}(b)(\a^{-5}(g_{1})\cdot\b^{-4}(c)))\sigma(\a^{-5}(g_{211}),\a^{-4}(l_{11})))]\}\\
&\sigma(\a^{-1}(h_{12}),\a^{-3}(g_{212})\a^{-2}(l_{12}))\#h_{2}(\a^{-2}(g_{22})\a^{-1}(l_2))\\
=&\b(a)\{[\a^{-3}(h_{11})\cdot((\b^{-3}(b)(\a^{-5}(g_{1})\cdot\b^{-4}(c)))\sigma(\a^{-5}(g_{211}),\a^{-4}(l_{11})))]\\
&\sigma(\a^{-2}(h_{12}),\a^{-4}(g_{212})\a^{-3}(l_{12}))\}\#h_{2}(\a^{-2}(g_{22})\a^{-1}(l_2))\\
=&\b(a)\{[\a^{-3}(h_{11})\cdot\b^{-2}((b(\a^{-4}(g_{11})\cdot\b^{-2}(c))\sigma(\a^{-3}(g_{12}),\a^{-4}(l_{1}))))]\\
&\sigma(\a^{-2}(h_{12}),\a^{-3}(g_{21})\a^{-3}(l_{21}))\}\#h_{2}(\a^{-2}(g_{22})\a^{-2}(l_{22}))\\
=&(\b(a)\#\a(h))[(b\#g)(c\#l)]
\end{aligned}$$
It is easy to see that $1\#1$ is the unit and $(\b\o\a)((a\#h)(b\#g))=(\b(a)\#\a(h))(\b(b)\#\a(g))$. The proof is completed.

\end{proof}

\begin{remark}
The condition (3) in the above proposition is actually a generalized form of Hom-2-cocycle introduced in \cite{LW}, when taking $A=k$.
\end{remark}

\begin{example}
\label{A}
(1) Consider the case when $\sigma$ is trivial, that is, $\sigma(h,g)=\varepsilon_H(h)\varepsilon_H(g)1$ for all $h,g\in H.$ Then the Hom-crossed product is reduced to Hom-smash product.

(2) Let $H$ be a Hopf algebra, $A$ an algebra and $H$ weakly acts on $A$. Assume that $\a$ is a Hopf automorphism of $H$ and $\b$ is an algebra isomorphism of $A$. Then we have the Hom-Hopf algebra $(H_\a,\a)$ and Hom-algebra $(A_\b,\b)$. Furthermore assume that $\b(h\cdot a)=\a(h)\cdot \b(a)$, then define the action $h\triangleright a=\a(h)\cdot \b(a)$, then $H_\a$ weakly acts on $A_\b$. If $A\#_\sigma H$ is a crossed product and $\sigma\circ(\a\o\a)=\b\circ\sigma$, $A_\b\#_\sigma H_\a$ is a Hom-crossed product.

(3) Let $H_4$ be a vector space with a basis $\{1,g,x,y\}$. Define the Hom-Hopf algebra structure on $H_4$ as follows:

multiplication:
$$\begin{array}{c|cccc}
 H_{4}&1&g&x&y\\
\hline 1&1&g&-x&-y\\
  g&g&1&-y&-x\\
  x&-x&y&0&0\\
  y&-y&x&0&0\\
 \end{array}$$
comultipication, counit and antipode:
\begin{eqnarray*}
&&\Delta(1)=1\otimes 1,\ \Delta(g)=g\otimes g,\\
&&\Delta(x)=(-x)\otimes g+1\otimes (-x),\ \Delta(y)=(-y)\o 1+g\o(-y),\\
&&\varepsilon(1)=1,\ \varepsilon(g)=1,\ \varepsilon(x)=0,\ \varepsilon(y)=0,\\
&&S(1_{H})=1_{H},\ S(g)=g,\ S(x)=y,\ S(y)=-x.
\end{eqnarray*}
The automorphism $\a:H_4\rightarrow H_4$ is given by
$$\alpha(1_{H})=1_{H},\ \alpha(g)=g,\ \alpha(x)=-x,\ \alpha(y)=-y.$$
Let $k[a]$ be the polynomial algebra with the indeterminant $a$ and $k[a]/(a^2)$ be the quotient algebra. Consider the Hom-algebra $(k[a]/(a^2),id)$ and define the action of $H_4$ on $k[a]/(a^2)$ by
$$h\cdot 1=\varepsilon(h)1,\ 1\cdot a=a,\ g\cdot a=a,\ x\cdot a=0,\ y\cdot a=0.$$

For any $t\in k$, define a linear map $\sigma:H_4\o H_4\rightarrow k[a]/(a^2)$ by
$$\begin{array}{|c| c c c c|}
\hline   \sigma  & 1 & g & x & y \\
\hline 1         & 1 & 1 & 0 & 0 \\
       g         & 1& 1 & 0 & 0 \\
       x         & 0 & 0 &  \frac{t}{2} & -\frac{t}{2}\\
       y        &  0  & 0 & \frac{t}{2} & -\frac{t}{2}\\
 \hline
 \end{array}$$
Easy to see that $\sigma$ satisfies the conditions in Proposition 2.2. Thus we have a crossed product $k[a]/(a^2)\#_\sigma H$ with the multiplication:
$$\begin{array}{|c| c c c c c c c c|}
\hline   \bullet     & 1\#1 & 1\#g & 1\#x & 1\#y & a\#1 & a\#g & a\#x & a\#y\\
\hline 1\#1         & 1\#1 & 1\#g & -1\#x  & -1\#y & a\#1 & a\#g & -a\#x & -a\#y\\
       1\#g         & 1\#g & 1\#1 & -1\#y & -1\#x & a\#g & a\#1 & -a\#y & -a\#x\\
       1\#x         & -1\#x & 1\#y &  0 & 0  &-a\#x & a\#y & -\frac{t}{2}a\#1 & -\frac{t}{2}a\#g\\
       1\#y         & -1\#y  & 1\#x & 0 & 0 &-a\#y & a\#x & -\frac{t}{2}a\#g & \frac{t}{2}a\#1\\
       a\#1         & a\#1   & a\#g & -a\#x & -a\#y & 0 & 0 & 0 & 0\\
       a\#g         & a\#g    & a\#1 &-a\#y & -a\#y & 0 & 0 & 0& 0\\
       a\#x         & -a\#x & a\#y & \frac{t}{2}a\#1 & -\frac{t}{2}a\#g & 0 & 0 & 0& 0 \\
       a\#y         & a\#y  & a\#x &\frac{t}{2}a\#g &-\frac{t}{2}a\#1 & 0 & 0 & 0& 0\\
 \hline
 \end{array}$$

\end{example}

\begin{lemma}
Especially we have the following identities:
\begin{itemize}
  \item [(1)]
  For all $h,g,l\in H$,
  \begin{eqnarray*}
  (i)~\lefteqn{h\cdot\sigma(g,l)=}\nonumber\\
  &[\sigma(\a^{-3}(h_{11}),\a^{-3}(g_{11}))\sigma(\a^{-4}(h_{12}l_{12}),\a^{-2}(l_1))]\sigma^{-1}(\a^{-1}(h_2),\a^{-2}(g_2l_2)),\\
(ii)~\lefteqn{h\cdot\sigma^{-1}(g,l)=}\nonumber\\
&\sigma(\a^{-1}(h_1),\a^{-2}(g_1l_1))[\sigma^{-1}(\a^{-4}(h_{21}g_{21}),\a^{-2}(l_2))\sigma^{-1}(\a^{-3}(h_{22}),\a^{-3}(g_{22}))].
\end{eqnarray*}
  \item [(2)] For all $a,b\in A$ and $h,g\in H$,
             \begin{itemize}
               \item [(i)] $(a\#1)(b\#1)=ab\#1$, $(1\#h)(1\#g)=\sigma(h_1,g_1)\#\a^{-1}(h_2g_2)$,
               \item [(ii)] $(1\#h)(a\#1)=\a^{-1}(h_1)\cdot a\#h_2$,  $(a\#1)(1\#h)=\b(a)\#\a(h)$.
\end{itemize}
\end{itemize}
\end{lemma}

\begin{proof}
Straightforward.
\end{proof}

Let $(H,\a)$ be a Hom-bialgebra and $(M,\mu)$ a right Hom-$H$-comodule via $\r$, then the coinvariant subcomodule $M^{coH}=\{m\in M|\r(m)=\mu(m)\o1\}$.

\begin{definition}
Let $(H,\a)$ be a Hom-bialgebra, $(B,\b)$ a right Hom-$H$-comodule algebra and $A=B^{coH}$. We say $A\subseteq B$ is a cleft extension if there exists a right $H$-comodule map $\g:H\rightarrow B$ which is convolution invertible.
\end{definition}

Without generalization we can assume $\g(1)=1$ in what follows.

\begin{lemma}
Assume that $(B,\beta)$ is a right $H$-Hom-comodule algebra, via $\rho:B\rightarrow B\otimes H,~b\mapsto b_{(0)}\o b_{(1)}$, and that $A\subset B$ is a $H$-cleft extension via $\g$. Then
\begin{itemize}
\item [(1)]
$\rho\circ\g^{-1}=(\g^{-1}\otimes S)\circ\tau\circ\Delta$
\item [(2)]
$b_{(0)}\g^{-1}(b_{(1)})\in A$ for any $b\in B.$
\end{itemize}
\end{lemma}

\begin{proof}
(1) Since $\rho$ is an algebra map, $\rho\circ\g^{-1}$ is the inverse of $\rho\circ\g=(\g\otimes id)\Delta.$ Let $\lambda=(\g^{-1}\otimes S)\circ\tau\circ\Delta$. Then for any $h\in H$
$$\begin{aligned}
((\rho\circ\g)\ast\lambda)(h)&=[(\g\otimes id)\Delta(h_{1})][(\g^{-1}\otimes S)\circ\tau\circ\Delta(h_{2})]\\
                                 &=[\g(h_{11})\otimes h_{12}][\g^{-1}(h_{22})\otimes S(h_{21})]\\
                                 &=\g(h_{11})\g^{-1}(h_{22})\otimes h_{12}S(h_{21})\\
                                 &=\g(\a(h_{1}))\g^{-1}(h_{22})\o\varepsilon_H(h_{21})1\\
                                 &=\g(\a(h_{1}))\g^{-1}(\a(h_{2}))\otimes 1\\
                                 &=\varepsilon_H(h)1\otimes 1.
\end{aligned}$$
Thus $\lambda=\rho\circ\g^{-1}$ by the uniqueness of inverses.

(2) For any $b\in B$
$$\begin{aligned}
\rho(b_{(0)}\g^{-1}(b_{(1)}))&=\rho(b_{(0)})\rho(\g^{-1}(b_{(1)}))\\
                                     &=(b_{(0)(0)}\otimes b_{(0)(1)})(\g^{-1}(b_{(1)2})\otimes S(b_{(1)1}))\\
                                     &= b_{(0)(0)}\g^{-1}(b_{(1)2})\otimes b_{(0)(1)}S(b_{(1)1})\\
                                     &= \beta(b_{(0)})\g^{-1}(b_{(1)2})\otimes\varepsilon_H(b_{(1)1})1\\
                                     &= \beta(b_{(0)})\g^{-1}(\alpha(b_{(1)}))\otimes 1\\                                   &=\beta(b_{(0)}\g^{-1}(b_{(1)}))\otimes 1.
\end{aligned}$$
The proof is completed.
\end{proof}

\begin{proposition}
Let $A\subset B$ be right $H$-cleft extension via $\g:H\rightarrow B$. Then there is a crossed product action of $H$ on $A$ given by
$$h\cdot a=(\g(\a^{-2}(h_{1}))\beta^{-1}(a))\g^{-1}(\alpha^{-1}(h_{2})),$$
and a convolution inverse map $\sigma:H\otimes H\rightarrow A$ given by
$$\sigma(h,g)=(\g(\a^{-3}(h_{1}))\g(\a^{-3}(g_{1})))\g^{-1}(\a^{-3}(h_{2}g_{2})).$$
Then we have the crossed product $A\#_{\sigma}H$. Moreover $\Phi:A\#_{\sigma}H\rightarrow B,\ a\#h\mapsto \b^{-2}(a)\g(\a^{-2}(h))$ is a Hom-algebra isomorphism. Moreover $\Phi$ is both a left $A$-module and right $H$-comodule map, where $a\cdot(b\#h)=\b(a)b\#\alpha(h)$ and $(a\#h)_{(0)}\o(a\#h)_{(1)}=\b(a)\#h_{1}\otimes\a^{-1}(h_{2})$.
\end{proposition}

\begin{proof}
Define the linear map $\Psi:B\rightarrow A\#_{\sigma}H$ by
$$b\mapsto\b^{-2}(b_{(0)(0)}\g^{-1}(b_{(0)(1)}))\#b_{(1)}.
$$
First of all we need to show $h\cdot a$ and $\sigma(h,g)$ belong to $A$ for all $a\in A,h,g\in H$. Indeed
$$\begin{aligned}
\r(h\cdot a)&=\r((\g(\a^{-2}(h_{1}))\beta^{-1}(a))\g^{-1}(\a^{-1}(h_{2})))\\
            &=[\r(\g(\a^{-2}(h_{1}))\r(\beta^{-1}(a)))] \r(\g^{-1}(\a^{-1}(h_{2})))\\
            &=[(\g(\a^{-2}(h_{11}))\otimes \a^{-2}(h_{12}))(a\otimes 1)](\g^{-1}(\a^{-1}(h_{22}))\o S(\a^{-1}(h_{21})))\\
              &=(\g(\a^{-2}(h_{11}))a)(\g^{-1}(\a^{-1}(h_{22})))\otimes\a^{-1}(h_{12})S(\a^{-1}(h_{21}))\\
              &=(\g(\a^{-1}(h_{1}))a)\beta(\g^{-1}(\a^{-1}(h_{22})))\otimes\varepsilon(h_{21})1\\
              &=(\g(\a^{-1}(h_{1}))a)\g^{-1}(h_{2})\otimes1\\
              &=\b(h\cdot a)\otimes1,
\end{aligned}$$
Thus $h\cdot a\in A$, and easy to see $H$ weakly acts on $A$.

And
$$\begin{aligned}
\rho(\sigma(h,g))&= (\rho\g(\a^{-3}(h_{1}))\rho\g(\a^{-3}(g_{1})))\rho\g^{-1}(\a^{-3}(h_2g_2))\\
                 &=(\b^{-3}\o\a^{-3})[(\g(h_{11})\otimes h_{12}))(\g(g_{11})\otimes g_{12})][\g^{-1}(h_{22}g_{22})\otimes S(h_{21}g_{21})]\\
                 &=(\b^{-3}\o\a^{-3})[(\g(h_{11})\g(g_{11}))\g^{-1}(h_{22}g_{22})\otimes (h_{12}g_{12})S(h_{21}g_{21})]\\
                 &=(\b^{-3}\o\a^{-3})[(\g(\a(h_{1}))\g(\a(g_{1})))\g^{-1}(\a(h_{2}g_{2}))\otimes 1]\\
                 &=\b(\sigma(h,g))\otimes 1.
\end{aligned}$$
Hence $\sigma(h,g)\in A.$

Next we show that $\Phi$ and $\Psi$ are mutual inverses. First for all $b\in B$,
$$\begin{aligned}
\Phi\Psi(b)&=\Phi(\b^{-2}(b_{(0)(0)}\g^{-1}(b_{(0)(1)}))\#b_{(1)})\\
           &=\b^{-4}(b_{(0)(0)}\g^{-1}(b_{(0)(1)})\g(\a^{-2}(b_{(1)}))\\
           &=\b^{-2}(b_{(0)})[\b^{-4}(\g^{-1}(b_{(1)1}))\g(\a^{-4}(b_{(1)2}))]\\
           &=b,
\end{aligned}$$
and for all $a\in A,h\in H$,
$$\begin{aligned}
\Psi\Phi(a\#h)&=\Psi(\b^{-2}(a)\g(\a^{-2}(h)))\\
              &=\b^{-2}[(a\g(\a^{-2}(h_{11})))\g^{-1}(\a^{-1}(h_{12}))]\#\a^{-1}(h_2)\\
              &=\b^{-2}[\b(a)(\g(\a^{-2}(h_{11}))\g^{-1}(\a^{-2}(h_{12})))]\#\a^{-1}(h_2)\\
              &=a\#h.
\end{aligned}$$
Thus $\Phi$ and $\Psi$ are mutual inverses. Moreover for all $a,b\in A$ and $h,g\in H$, by a direct computation
\begin{eqnarray}
&&(\a^{-4}(h_{1})\cdot \b^{-2}(b))\sigma(\a^{-3}(h_{2}),\a^{-2}(g))\nonumber\\
=&&[\g(\a^{-4}(h_1))\b^{-2}(b)][\g(\a^{-4}(g_1))\g^{-1}(\a^{-5}(h_2g_2))],
\end{eqnarray}
then
$$\begin{aligned}
&~\Phi((a\#h)(b\#g))\\
&=\b^{-2}\{a[(\a^{-4}(h_{11})\cdot \b^{-2}(b))\sigma(\a^{-3}(h_{12}),\a^{-2}(g_{1}))]\}\g(\a^{-3}(h_2g_2))\\
&\stackrel{(2.1)}{=}\b^{-2}\{a[[\g(\a^{-4}(h_{11}))\b^{-2}(b)][\g(\a^{-4}(g_{11}))\g^{-1}(\a^{-5}(h_{12}g_{12}))]]\}\g(\a^{-3}(h_2g_2))\\
&=\b^{-2}\{a[[(\g(\a^{-5}(h_{11}))\b^{-3}(b))\g(\a^{-4}(g_{11}))][\g^{-1}(\a^{-4}(h_{12}g_{12}))]]\}\g(\a^{-3}(h_2g_2))\\
&=\{\b^{-2}(a)[(\g(\a^{-6}(h_{11}))\b^{-4}(b))\g(\a^{-5}(g_{11}))]\}\{\g^{-1}(\a^{-5}(h_{12}g_{12}))\g(\a^{-4}(h_2g_2))\}\\
&=\b^{-1}(a)[(\g(\a^{-3}(h))\b^{-3}(b))\g(\a^{-2}(g))]\\
&=[\b^{-2}(a)\g(\a^{-2}(h))][\b^{-2}(b)\g(\a^{-2}(g))]\\
&=\Phi(a\#h)\Phi(b\#g).
\end{aligned}$$
Finally we need to check that $\Phi$ is both a left $A$-module and right $H$-comodule map. For all $h\in H$ and $a,b\in A$,
$$\begin{aligned}
\Phi(a\cdot b\#h)=(\b^{-1}(a)\b^{-2}(b))\g(\a^{-1}(h))=a(\b^{-2}(b)\g(\a^{-2}(h)))=a\cdot\Phi(b\#h),
\end{aligned}$$
and
$$\begin{aligned}
\Phi(a\#h)_{(0)}\o\Phi(a\#h)_{(1)}&=\b^{-1}(a)\g(\a^{-2}(h_1))\o\a^{-1}(h_2)\\
                                  &=\Phi(\b(a)\#h_1)\o\a^{-1}(h_2)\\
                                  &=\Phi((a\#h)_{(0)})\o(a\#h)_{(1)}.
\end{aligned}$$
The proof is completed.
\end{proof}

\begin{proposition}
Let $(A\#_\sigma H,\b\o\a)$ be a Hom-crossed product. Define the map $\g:H\rightarrow A\#_\sigma H$ by $\g(h)=1\#\a(h)$. Then $\g$ is convolution invertible right $H$-comodule map.
\end{proposition}

\begin{proof}
First of all for all $h\in H$, $\g\circ\a=(\b\o\a)\circ\g$ and
$$\g(h)_{(0)}\o\g(h)_{1}=1\#\a(h_1)\o h_2=\g(h_1)\o h_2,$$
which means that $\g$ is right $H$-comodule map.

Define a linear map $\l:H\rightarrow A\#_\sigma H$ by
$$\l(h)=\sigma^{-1}(S\a^{-1}(h_{21}),\a^{-1}(h_{22}))\#S(h_1).$$
Now we verify that $\l$ is the convolution inverse of $\g$.

For all $h\in H$,
$$\begin{aligned}
(\l\ast\g)(h)&=(\sigma^{-1}(S\a^{-1}(h_{121}),\a^{-1}(h_{122}))\#S(h_{11}))(1\#\a(h_2))\\
             &=\sigma^{-1}(S\a^{-1}(h_{121}),\a^{-1}(h_{122}))\sigma(S\a^{-1}(h_{112}),h_{21})\#S\a^{-1}(h_{111})h_{22}\\
             &=\varepsilon(h_{21})\varepsilon(h_{221})1\#S\a(h_{1})\a^{-1}(h_{222})\\
             &=\varepsilon(h)1\#1,
\end{aligned}$$
and
$$\begin{aligned}
&(\g\ast\l)(h)\\
=&(1\#\a(h_1))(\sigma^{-1}(S\a^{-1}(h_{221}),\a^{-1}(h_{222}))\#S(h_{22}))\\
=&[\a^{-2}(h_{111})\cdot\sigma^{-1}(S\a^{-2}(h_{221}),\a^{-2}(h_{222}))]\sigma(\a^{-1}(h_{12}),S\a^{-1}(h_{212}))\#h_{12}S\a^{-1}(h_{211})\\
=&[\a^{-1}(h_{11})\cdot\sigma^{-1}(S\a^{-2}(h_{221}),\a^{-2}(h_{222}))]\sigma(h_{12},S(h_{21}))\#1\\
=&[\sigma(\a^{-2}(h_{111}),\a^{-4}(S(h_{2212})h_{2221}))(\sigma^{-1}(\a^{-5}(h_{1121})S\a^{-6}(h_{22112}),\a^{-4}(h_{2222}))\\
&\sigma^{-1}(\a^{-4}(h_{1122}),S\a^{-5}(h_{22111})))]\sigma(h_{12},S(h_{21}))\#1\\
=&[\sigma(\a^{-2}(h_{111}),\a^{-5}(S(h_{22211})h_{22212}))\sigma^{-1}(\a^{-4}(h_{1121})S\a^{-4}(h_{2212}),\a^{-3}(h_{2222}))]\\
&[\sigma^{-1}(\a^{-3}(h_{1122}),S\a^{-3}(h_{2211}))\sigma(\a^{-1}(h_{12}),S\a^{-1}(h_{21}))]\#1\\
=&\sigma^{-1}(\a^{-2}(h_{111})S\a^{-3}(h_{2212}),\a^{-1}(h_{222}))\\
&[\sigma^{-1}(\a^{-2}(h_{112}),S\a^{-3}(h_{2211}))\sigma(\a^{-1}(h_{12}),S\a^{-1}(h_{21}))]\#1\\
=&\sigma^{-1}(\a^{-1}(h_{11})S\a^{-2}(h_{212}),h_{22})\varepsilon(h_{12})\varepsilon(h_{211})\\
=&\sigma^{-1}(h_{1}S\a^{-1}(h_{21}),h_{22})\#1\\
=&\varepsilon(h)1\#1,
\end{aligned}$$
where the fourth identity is obtained by using Lemma 2.5 (1).
The proof is completed.

\end{proof}

Note that if $\sigma$ is trivial the crossed product $A\#_\sigma H$ is reduced to smash product $A\#H$. Then $\g$ is invertible with $\g^{-1}(h)=1\#S\a(h)$.

By the above two propositions, we have the main theorem of this section.

\begin{theorem}
If the extension $A\subseteq B$ is cleft if and only if $B\simeq A\#_\sigma H$.
\end{theorem}

\section{Hom-Hopf algebra structure on $A\#_\sigma H$}
\def\theequation{3.\arabic{equation}}
\setcounter{equation} {0} \hskip\parindent

In this section we will give the necessary and sufficient conditions which make the Hom-crossed product into a Hom-Hopf algebra.

\begin{definition}
Let $(A\t_\sigma H,\b\o\a)$ be a Hom-crossed product, then $\sigma$ is called a twisted comodule cocycle if
\begin{eqnarray}
\b(a_1)\o\a^{-1}(a_{2(-1)})\a(g)\o a_{2(0)}=a_1\sigma(\a^{-2}(a_{2(-1)1}),g_1)\o\a^{-2}(a_{2(-1)2})g_2\o a_{2(0)}
\end{eqnarray}
for all $a\in A$ and $g\in H$.
\end{definition}

Note that when $\sigma$ is trivial, the definition is natural. Then we give the main result of this section.

\begin{proposition}
Let $(H,\a)$ be a Hom-bialgebra and $(A,\b)$ a Hom-algebra and Hom-coalgebra. Suppose that $H$ weakly acts on $A$ and $A$
is a left $H$-Hom-comodule coalgebra with the comodule structure map $\r:A\rightarrow H\o A, a\mapsto a_{(-1)}\o a_{(0)}$. Suppose that $(A\#_\sigma H,\b\o\a)$ is a Hom-crossed product with $\sigma$ being a twisted comodule cocycle, and $(A\#_\sigma H,\b\o\a)$ is a Hom-smash coproduct, then the following conditions are equivalent:
\begin{itemize}
  \item[(1)] $(A\#_\sigma H,\b\o\a)$ is a Hom-bialgebra.
  \item[(2)] The conditions:
  \begin{itemize}
    \item [$A_1.$] \ $\varepsilon_A$ is an Hom-algebra map,
    \item [$A_2.$] \ $\varepsilon_A(h\cdot a)=\varepsilon_H(h)\varepsilon_A(a)$,
    \item [$A_3.$] \ $\sigma$ is a Hom-coalgebra map.
    \item [$A_4.$] \ $\Delta_A(h\cdot a)=(\a^{-2}(h_{11})\cdot\b^{-1}(a_1))\sigma(\a^{-1}(h_{12}),\a^{-1}(a_{2(-1)}))\o h_2\cdot \b^{-1}(b_{2(0)})$,
    \item [$A_5.$] \ $(\a^{-1}(h_1)\cdot a)_{(-1)}\a(h_2)\o (\a^{-1}(h_1)\cdot a)_{(0)}=\a(h_1a_{(-1)})\o h_2\cdot b_{(0)}$,
    \item [$A_6.$] \ $\Delta_A(ab)=a_1[(\a^{-4}(a_{2(-1)1})\cdot\b^{-2}(b_1))\sigma(\a^{-3}(a_{2(-1)2}),\a^{-2}(b_{2(-1)}))]\o \b^{-1}(a_{2(0)}b_{2(0)})$,
    \item [$A_7.$] \ $\sigma(h_1,g_1)_{(-1)}(h_2g_2)\o\sigma(h_1,g_1)_{(0)}=\a(h_1g_1)\o\sigma(\a(h_2),\a(g_2))$,
    \item [$A_8.$] \ $\Delta_A(1)=1\o1,$
    \item [$A_9.$] \ $\r(ab)=\r(a)\r(b),~~\r(1)=1\o 1$.
  \end{itemize}
\end{itemize}
\end{proposition}

\begin{proof}
$(1)\Rightarrow(2)$ follows from the similar calculations to those of \cite[Theorem 1]{R}. So we need only to show $(2)\Rightarrow(1)$. Assume (2) holds, then by $A_1$ and $A_2$, $\varepsilon$ is a Hom-algebra map. By $A_8$ and $A_9$,
$\Delta(1\#1)=1\#1\o1\#1$. In order to prove $\Delta((a\o h)(b\o g))=\Delta(a\o h)\Delta(b\o g)$, it is enough to verify the following relations:
\begin{eqnarray*}
&&(i)\ \Delta((a\o1)(b\o1))=\Delta(a\o1)\Delta(b\o1),\\
&&(ii)\ \Delta((a\o1)(1\o g))=\Delta(a\o1)\Delta(1\o g),\\
&&(iii)\ \Delta((1\o h)(b\o1))=\Delta(1\o h)\Delta(b\o1),\\
&&(iv)\ \Delta((1\o h)(1\o g))=\Delta(1\o h)\Delta(1\o g).
\end{eqnarray*}
In fact
$$\begin{aligned}
&\Delta(a\o1)\Delta(b\o1)\\
=&[a_1\t\a^{-1}(a_{2(-1)})\o\b^{-1}(a_{2(0)})\t 1][b_1\t\a^{-1}(b_{2(-1)})\o\b^{-1}(b_{2(0)})\t 1]\\
=&a_1[(\a^{-5}(a_{2(-1)11})\cdot \b^{-2}(b_1))\sigma(\a^{-4}(a_{2(-1)12}),\a^{-3}(b_{2(-1)1}))]\#\a^{-2}(a_{2(-1)2}b_{2(-1)2})\\
&\o\b^{-1}(a_{2(0)}b_{2(0)})\#1\\
=&a_1[(\a^{-4}(a_{2(-1)1})\cdot \b^{-2}(b_1))\sigma(\a^{-3}(a_{2(-1)2}),\a^{-2}(b_{2(-1)}))]\#\a^{-2}(a_{2(0)(-1)}b_{2(0)(-1)})\\
&\o\b^{-2}(a_{2(0)(0)}b_{2(0)(0)})\#1\\
\stackrel{A_6}{=}&(ab)_1\#\a^{-1}((ab)_{2(-1)})\o\b^{-1}((ab)_{2(0)})\#1\\
=&\Delta(ab\#1),
\end{aligned}$$
and $(i)$ is proved.

$$\begin{aligned}
\Delta(a\o1)\Delta(1\o g)=&a_1\sigma(\a^{-2}(a_{2(-1)1}),\a^{-1}(g_{11}))\#\a^{-2}(a_{2(-1)2})\a^{-1}(g_{12})\o a_{2(0)}\#\a(g_2)\\
\stackrel{(3.1)}{=}&\b(a_1)\#\a^{-1}(a_{2(-1)})g_1\o a_{2(0)}\#\a(g_2)\\
=&\Delta((a\#1)(1\# g)),
\end{aligned}$$
and $(ii)$ is proved.

$$\begin{aligned}
&\Delta((1\o h)(b\o1))\\
=&(\a^{-1}(h_1)\cdot b)_1\#\a^{-2}((\a^{-1}(h_1)\cdot b)_{2(-1)})\a^{-1}(h_{21})\o\b^{-1}((\a^{-1}(h_1)\cdot b)_{2(0)})\#h_{22}\\
\stackrel{A_4}{=}&(\a^{-3}(h_{111})\cdot\b^{-1}(b_1))\sigma(\a^{-2}(h_{112}),\a^{-1}(b_{2(-1)}))\#\a^{-2}[(\a^{-1}(h_{12})\cdot \b^{-1}(b_{2(0)}))_{(-1)}\a(h_{21})]\\
&\o\b^{-1}((\a^{-1}(h_{12})\cdot \b^{-1}(b_{2(0)}))_{(0)})\#h_{22}\\
=&(\a^{-2}(h_{11})\cdot\b^{-1}(b_1))\sigma(\a^{-1}(h_{12}),\a^{-1}(b_{2(-1)}))\#\a^{-2}[(\a^{-2}(h_{211})\cdot \b^{-1}(b_{2(0)}))_{(-1)}h_{212}]\\
&\o\b^{-1}((\a^{-2}(h_{211})\cdot \b^{-1}(b_{2(0)}))_{(0)})\#h_{22}\\
\stackrel{A_5}{=}&(\a^{-2}(h_{11})\cdot\b^{-1}(b_1))\sigma(\a^{-1}(h_{12}),\a^{-2}(b_{2(-1)1}))\#\a^{-2}(h_{211}b_{2(-1)2})\\
&\o\b^{-1}(\a^{-1}(h_{212})\cdot b_{2(0)})\#h_{22}\\
=&(\a^{-3}(h_{111})\cdot\b^{-1}(b_1))\sigma(\a^{-2}(h_{112}),\a^{-2}(b_{2(-1)1}))\#\a^{-1}(h_{12})\a^{-2}(b_{2(-1)2})\\
&\o\b^{-1}(h_{21}\cdot b_{2(0)})\#h_{22}\\
=&\Delta(1\o h)\Delta(b\o1),
\end{aligned}$$
and $(iii)$ is proved.
$$\begin{aligned}
&\Delta(1\o h)\Delta(1\o g)\\
=&(1\#h_1)(1\#g_1)\o(1\#h_2)(1\#g_2)\\
=&\sigma(h_{11},g_{11})\#\a^{-1}(h_{12}g_{12})\o\sigma(h_{21},g_{21})\#\a^{-1}(h_{22}g_{22})\\
=&\sigma(h_{11},g_{11})\#\a^{-2}(\sigma(h_{12},g_{12})_{(-1)})\a^{-2}(h_{21}g_{21})\o\b^{-1}(\sigma(h_{12},g_{12})_{(0)})\#\a^{-1}(h_{22}g_{22})\\
\stackrel{A_3}{=}&\Delta((1\#h)(1\#g)),
\end{aligned}$$
and $(iv)$ is proved.

The proof is completed.

\end{proof}

\begin{definition}
Let $(H,\a)$ be a Hom-bialgebra, $(A,\b)$ a Hom-algebra and $\sigma:H\o H\rightarrow A$ and $S:H\rightarrow H$ a linear map. $S$ is called a $\sigma$-antipode of $H$ if
\begin{itemize}
  \item [(i)] $\a\circ S=S\circ\a$,
  \item [(ii)] $(\sigma\o m_H)\Delta_{H\o H}(id\o S)\Delta_H(h)=\varepsilon_H(h)1\o1$,
  \item [(iii)] $(\sigma\o m_H)\Delta_{H\o H}(S\o id)\Delta_H(h)=\varepsilon_H(h)1\o1$.
\end{itemize}
In this case $H$ is called a $\sigma$-Hom Hopf algebra.
\end{definition}

\begin{example}
Let $(H,\a)$ be a Hom-Hopf algebra. Consider the case when $\sigma$ is trivial, then we can regard $S_H$ as $\sigma$-antipode of $H$.
\end{example}

\begin{proposition}
$(A\#_\sigma H,\b\o\a)$ be a Hom-bialgebra. If $(H,\a)$ is a $\sigma$-Hom Hopf algebra with the $\sigma$-antipode $S_H$ and $S_A\in Hom(A,A)$ is a convolution invertible element of $id_A$ with $\b\circ S_A=S_A\circ\b$. Then $(A\#_\sigma H,\b\o\a)$ is a Hom-Hopf algebra with the antipode given by
$$S(a\#h)=(1\#S_H(\a^{-3}(a_{(-1)})\a^{-2}(h)))(S_A(\b^{-2}(a_{(0)}))\#1)$$
\end{proposition}

\begin{proof}
Firstly for all $a\in A$ and $h\in H$,
$$S\circ(\b\o\a)(a\#h)=(\b\o\a)\circ S(a\#h).$$
And
$$\begin{aligned}
&S(a_1\#\a^{-2}(a_{2(-1)})\a^{-1}(h_1))(\b^{-1}(a_{2(0)})\#h_2)\\
=&(1\#S_H(\a^{-2}(a_{1(-1)})(\a^{-3}(a_{2(-1)})\a^{-2}(h_1))))[(S_A(\b^{-2}(a_{1(0)}))\#1)(\b^{-2}(a_{2(0)})\#\a^{-1}(h_2))]\\
=&[1\#S_H(\a^{-3}(a_{1(-1)}a_{2(-1)})\a^{-1}(h_1))][S_A(\b^{-2}(a_{1(0)}))\b^{-2}(a_{2(0)})\#h_2]\\
=&[1\#S_H(\a^{-1}(a_{(-1)})\a^{-1}(h_1))][S_A(\b^{-2}(a_{(0)1}))\b^{-2}(a_{(0)2})\#h_2]\\
=&\varepsilon(a)[1\#S_H(h_1)][1\#h_2]\\
=&\varepsilon(a)\sigma(S_H(h_1)_1,h_{21})\#\a^{-1}(S_H(h_1)_2h_{22})\\
=&\varepsilon(a)\varepsilon(h)1\#1.
\end{aligned}$$
Similarly we can verify that $(a_1\#\a^{-2}(a_{2(-1)})\a^{-1}(h_1))S(\b^{-1}(a_{2(0)})\#h_2)=\varepsilon(a)\varepsilon(h)1\#1.$
Hence $S$ is the convolution inverse of $id$.

The proof is completed.

\end{proof}

\begin{corollary}
With the above notations, if $\sigma$ is trivial, we have the Radford biproduct $(A\t H,\b\o\a)$. At this moment we call $(H,A)$ is an admissible pair.
\end{corollary}

\begin{remark}
If $H$ is a Hopf algebra, and $(H,A)$ is an admissible pair, it is well known that $A$ is a Hopf algebra in the Yetter-Drinfeld category $_HYD(H)^H$. However to our disappointment in the case of Hom-Hopf algebra, this conclusion does not hold unless $\a^2=id_H$.
\end{remark}

\begin{example}
In the Example \ref{A} (3), consider the crossed product $k\#_\sigma H$ and taking $t=0$, then $k\#_\sigma H$ is a Hom-Hopf algebra and $k\#_\sigma H=H$ as a Hom-Hopf algebra.
\end{example}

\section{Lazy 2-cocycle}
\def\theequation{4.\arabic{equation}}
\setcounter{equation} {0} \hskip\parindent

In this section we will generalize the theory of lazy 2-cocycle to Hom-Hopf algebras.
Recall from \cite{LW} that a left 2-cocycle on a Hom-bialgebra $(H,\a)$ is a linear map $\sigma:H\o H\rightarrow k$ satisfying
\begin{eqnarray*}
&&\sigma\circ(\a\o\a)=\sigma,\\
&&\sigma(l_{1},k_{1})\sigma(\a^{2}(h),l_{2}k_{2})=\sigma(h_{1},l_{1})\sigma(h_{2}l_{2},\a^{2}(k)),
\end{eqnarray*}
for all $h,k,l\in H$.

$\sigma$ is a right Hom-2-cocycle if
\begin{eqnarray*}
&&\sigma\circ(\a\o\a)=\sigma,\\
&&\sigma(\a^{2}(h),l_{1}k_{1})\sigma(l_{2},k_{2})=\sigma(h_{1}l_{1},\a^{2}(k))\sigma(h_{2},l_{2}).
\end{eqnarray*}

$\sigma$ is called {\sl normal} if $\sigma(1,h)=\sigma(h,1)=\varepsilon(h).$

If $\sigma$ is normalized and convolution invertible, then $\sigma$ is a left Hom-2-cocycle if and only if $\sigma^{-1}$ is a right Hom-2-cocycle.

Given a linear map $\sigma:H\o H\rightarrow k$, define a new multiplication on $H$ by
$$h\cdot_{\sigma}g=\sigma(h_1,g_1)\a^{-1}(h_2g_2).$$
Then $\cdot_{\sigma}$ is Hom-associative if and only if $\sigma$ is a left Hom-2-cocycle.

If we define the multiplication by
$$h\ _{\sigma}\cdot g=\a^{-1}(h_{1}g_{1})\sigma(h_{2},g_{2}).$$
Then $_{\sigma}\cdot$ is Hom-associative if and only if $\sigma$ is a right Hom-2-cocycle.

A left Hom-2-cocycle $\sigma$ is called lazy if for all $h,g\in H$
$$\sigma(h_1,g_1)h_2g_2=h_1g_1\sigma(h_2,g_2).$$
A lazy left Hom-2-cocycle is also a right Hom-2-cocycle.

\begin{example}
Note that $\sigma$ defined in Example \ref{A} (3) is a lazy 2-cocycle on $H_4$.
\end{example}

\begin{lemma}
Let $\g:H\rightarrow k$ be a normalized and convolution invertible linear map such that $\g\circ\a=\g$, define $D^1(\g):H\o H\rightarrow k$ by
$$D^1(\g)(h,g)=\g(h_1)\g(g_1)\g^{-1}(h_2g_2),$$
for all $h,g\in H$. Then $D^1(\g)$ is a normalized and convolution invertible left 2-cocycle.
\end{lemma}

\begin{proof}
This is an easy consequence of Proposition 2.8.

\end{proof}

$\g$ is lazy if for all $h\in H$, $\g(h_1)h_2=h_1\g(h_2).$

The set of all normalized and convolution invertible linear maps $\g:H\rightarrow k$ satisfying $\g\circ\a=\g$ is denoted by $Reg^1_L(H)$, which is a group under convolution.

\begin{lemma}
The set of convolution invertible lazy Hom-2-cocycle denoted by $Z^2_L(H)$ is a group.
\end{lemma}

\begin{proof}
Suppose that $\sigma_1,\sigma_2\in Z^2_L(H)$, and for all $h,g,l\in H$,
$$\begin{aligned}
&(\sigma_1\ast\sigma_2)(g_1,l_1)(\sigma_1\ast\sigma_2)(\a^2(h),g_2l_2)\\
=&\sigma_1(g_{11},l_{11})\sigma_2(g_{12},l_{12})\sigma_1(\a^2(h_1),g_{21}l_{21})\sigma_2(\a^2(h_2),g_{22}l_{22})\\
=&\sigma_1(g_{11},l_{11})\sigma_1(\a^2(h_1),g_{12}l_{12})\sigma_2(g_{21},l_{21})\sigma_2(\a^2(h_{2}),g_{22}l_{22})\\
=&\sigma_1(h_{11},g_{11})\sigma_1(h_{12}g_{12},\a^2(l_1))\sigma_2(h_{21},g_{21})\sigma_2(h_{22}g_{22},\a^2(l_2))\\
=&(\sigma_1\ast\sigma_2)(h_1,g_1)(\sigma_1\ast\sigma_2)(h_2g_2,\a^2(l)),
\end{aligned}$$
thus $\sigma_1\ast\sigma_2$ is a left 2-cocycle on $H$, and it is easy to see that $\sigma_1\ast\sigma_2$ is lazy. The proof is completed.

\end{proof}

\begin{proposition}
The map $D^{1}:Reg^{1}_{L}(H)\rightarrow Z^{2}_{L}(H)$ is a group homomorphism,
whose image denoted by $B^{2}_{L}(H)$, is contained in the center of $Z^{2}_{L}(H)$. Thus we call quotient group $H^{2}_{L}(H,\alpha):=Z^{2}_{L}(H,\alpha)/B^{2}_{L}(H,\alpha)$ the second lazy cohomology group of $H$.
\end{proposition}

\begin{proof}
For all $\g_1,\g_2\in Reg^{1}_{L}(H)$ and $h,g\in H$,
$$\begin{aligned}
D^1(\g_1\ast\g_2)(h,g)=&\g_1(h_{11})\g_2(h_{12})\g_1(g_{11})\g_2(g_{12})\g_2^{-1}(h_{21}g_{21})\g_1^{-1}(h_{22}g_{22})\\
=&\g_1(h_{1})\g_1(g_{1})\g_1^{-1}(h_{22}g_{22})\g_2(h_{211})\g_2(g_{211})\g_2^{-1}(h_{212}g_{212})\\
=&\g_1(h_{1})\g_1(g_{1})\g_1^{-1}(h_{22}g_{22})D^{1}(\g_2)(h_{21},g_{21})\\
=&\g(h_{1})\g_1(g_{1})\g_1^{-1}(h_{21}g_{21})D^{1}(\g_2)(h_{22},g_{22})\\
=&\g_1(h_{11})\g_1(g_{11})\g_1^{-1}(h_{12}g_{12})D^{1}(\g_2)(h_{2},g_{2})\\
=&(D^{1}(\g_1)\ast D^{1}(\g_2))(h,g),
\end{aligned}$$
and $D^{1}(\varepsilon_H)=\varepsilon_H\o\varepsilon_H.$ Thus $D^1$ is a group homomorphism.

For all $\g\in Reg^{1}_{L}(H)$ and $\sigma\in Z^{2}_{L}(H)$,
$$\begin{aligned}
(\sigma\ast D^{1}(\g))(h,g)
=&\sigma(h_{1},g_{1})\g(h_{21})\g(g_{21})\g^{-1}(h_{22}g_{22})\\
=&\sigma(h_{12},g_{12})\g(h_{2})\g(g_{2})\g^{-1}(h_{11}g_{11})\\
=&\sigma(h_{21},g_{21})\g(h_{22})\g(g_{22})\g^{-1}(h_{1}g_{1})\\
=&\sigma(h_{22},g_{22})\g(h_{21})\g(g_{21})\g^{-1}(h_{1}g_{1})\\
=&\g(h_{11})\g(g_{11})\g^{-1}(h_{12}g_{12})\sigma(h_{2},g_{2})\\
=&(D^{1}(\g)\ast\sigma)(h,g).
\end{aligned}$$
The proof is completed.
\end{proof}

\begin{lemma}
Let $\sigma:H\o H\rightarrow k$ be a normalized and convolution invertible left (respectively right) Hom-2-cocycle, $H_{\sigma}$ (respectively $_{\sigma}H$) is right (respectively left) $H$-comodule algebra via $\Delta$. If $\sigma$ is lazy, $H_{\sigma}=\ _{\sigma}H$ as algebras and we denote it by $H(\sigma)$. Moreover $H(\sigma)$ is an $H$-bicomodule algebra.
\end{lemma}

\begin{proof}
Straightforward.

\end{proof}

In the following lemma, we will list the formulae useful in our computations.

\begin{lemma}
\begin{itemize}
\item [(1)] Let $\sigma$ be a normalized and convolution invertible left Hom-2-cocycle. For all $h\in H$
\begin{eqnarray}
&&\sigma(h_{11},S(h_{12}))\sigma^{-1}(S(h_{21}),h_{22})=\varepsilon(h),\\
&&\sigma(S^{-1}(h_{12}),h_{11})\sigma^{-1}(h_{22},S^{-1}(h_{21}))=\varepsilon(h),\\
&&\sigma(h_{11},g_{11})\sigma(h_{12}g_{12},S(\a(h_2g_2)))\nonumber\\ &&~~~=\sigma(g_{11},S(g_{12}))\sigma(h_{11},S(h_{12}))\sigma^{-1}(S(g_2),S(h_2)).
\end{eqnarray}
\item [(2)]
If $\sigma$ is lazy, we have the following relations:
\begin{eqnarray}
&&\sigma(h_{1},S(h_{2}))=\sigma(S(h_1),h_2),\\
&&\sigma(S^{-1}(h_{2}),h_{1})=\sigma^{-1}(h_{2},S^{-1}(h_{1})),\\
&&\sigma^{-1}(h_{21},S^{-1}(h_{12}))h_{22}S^{-1}(h_{11})=\sigma^{-1}(h_2,S^{-1}(h_1))1,\\
&&\sigma^{-1}(S^{-1}(h_{21}),h_{12})S^{-1}(h_{22})h_{11}=\sigma^{-1}(S^{-1}(h_2),h_1)1,\\
&&\sigma^{-1}(S(h_{12}),h_{21})S(h_{11})h_{22}=\sigma^{-1}(S(h_1),h_2)1,\\
&&\sigma^{-1}(S(h_{21}),h_{22})S(h_{1})=\sigma^{-1}(S(h_{11}),h_{12})S(h_2),\\
&&\sigma^{-1}(h_{12},S(h_{21}))h_{11}S(h_{22})=\sigma^{-1}(h_1,S(h_2))1,\\
&&\sigma^{-1}(S(h_{12}),h_{21})h_{22}S^{-1}(h_{11})=\sigma^{-1}(S(h_1),h_2)1.
\end{eqnarray}
\end{itemize}
\end{lemma}

\begin{proof}
For the proof we could refer to \cite{CP}.
\end{proof}

\begin{proposition}
Define the linear maps $S_1,S_2:H\rightarrow H$ by
\begin{eqnarray*}
&&S_1(h)=\sigma^{-1}(S(h_{21}),h_{22})S(\a^{-1}(h_1)),\\
&&S_2(h)=\sigma^{-1}(h_{22},S^{-1}(h_{21}))S^{-1}(\a^{-1}(h_1)).
\end{eqnarray*}
If $\sigma$ is lazy then $S_1,S_2:H(\sigma^{-1})\rightarrow H(\sigma)$ are Hom-algebra anti-isomorphisms, and for all $h\in H$,
\begin{eqnarray}
&&S_1(h_1)\cdot_\sigma h_2=\varepsilon(h)1=h_1\cdot_\sigma S_1(h_2),\\
&&S_2(h_2)\cdot_\sigma h_1=\varepsilon(h)1=h_2\cdot_\sigma S_2(h_1),\\
&&\Delta(S_1(h))=S_1(h_2)\o S(h_1),\\
&&\Delta(S_2(h))=S_2(h_2)\o S^{-1}(h_1).
\end{eqnarray}
\end{proposition}

\begin{proof}
Define a linear map $\phi_\sigma:H(\sigma)\rightarrow H(\sigma^{-1})$ by
$$\phi_\sigma(h)=\sigma(h_{11},S(h_{12}))S(\a^{-1}(h_2)).
$$
Then by (4.2) and (4.3) we know $S_2$ and $\phi_\sigma$ are mutual inverses. If $\sigma$ is lazy, by (4.10) $S_1=\phi_{\sigma^{-1}}$. Hence $S_1$ is invertible. Finally by $(4.4)$ we obtain that $S_1$ and $S_2$ are anti-isomorphisms.

The rest of the proof is an easy exercise.

\end{proof}

Let $(H,\a)$ be a Hom-Hopf algebra. An $H$-bicomodule $(A,\b)$ is both a left and right $H$-comodule with comodule structures $A\rightarrow A\o H,~a\mapsto a_{(0)}\o a_{(1)}$ and $A\rightarrow H\o A,~a\mapsto a_{[-1]}\o a_{[0]}$ such that for all $a\in A$,
$$\a(a_{[-1]})\o a_{[0](0)}\o a_{[0](1)}=a_{(0)[-1]}\o a_{(0)[0]}\o \a(a_{(1)}).$$

Recall from \cite{MP} that the diagonal crossed product $H^*\bowtie A$ is equal to $H^*\o A$ as vector space with the multiplication
$$(p\bowtie a)(q\bowtie b)=p[(\a^{-3}(a_{[-1]})\rightharpoonup \a^{*2}(q))\leftharpoonup\a^{-3}(S^{-1}(a_{[0](1)}))]\bowtie\a^{-2}(a_{[0](0)})b,$$
for all $p,q\in H^*$ and $a,b\in A$.

Furthermore the space $H^*\bowtie A$ becomes a $D(H)$-bicomodule algebra with the following structures:
\begin{eqnarray*}
&&H^*\bowtie A\rightarrow H^*\bowtie A\o D(H),~~p\bowtie a\mapsto (p_2\bowtie a_{(0)})(p_1\o a_{(1)}),\\
&&H^*\bowtie A\rightarrow D(H)\o H^*\bowtie A,~~p\bowtie a\mapsto (p_2\bowtie a_{[-1]})(p_1\o a_{[0]}).
\end{eqnarray*}

Let $\sigma:H\o H\rightarrow k$ be a normalized and convolution invertible left lazy Hom-2-cocycle, then $\bar{\sigma}:D(H)\o D(H)\rightarrow k$ given by
$$\bar{\sigma}(p\o h,q\o g)=p(1)q(S^{-1}(\a^{-2}(h_{22}))\a^{-1}(h))\sigma(h_{21},\a^{2}(g))$$
is a normalized and convolution invertible lazy Hom-2-cocycle with the convolution inverse $$\bar{\sigma}^{-1}(h,g)=p(1)q(S^{-1}(\a^{-2}(h_{22}))\a^{-1}(h))\sigma^{-1}(h_{21},\a^{2}(g)).$$

\begin{proposition}
Let $\sigma:H\o H\rightarrow k$ be a normalized and convolution invertible left lazy Hom-2-cocycle. Consider the $H$-bicomodule algebra $H(\sigma)$. Then $H^*\bowtie H(\sigma)=D(H)(\bar{\sigma})$ as $D(H)$-bicomodule algebras. Moreover $\bar{\sigma}$ is unique with this property.
\end{proposition}

\begin{proof}
For all $h,g\in H$ and $p,q\in H^*$,
$$\begin{aligned}
&(p\bowtie h)(q\bowtie g)\\
=&p[(\a^{-3}(h_1)\rightharpoonup \a^{*2}(q))\leftharpoonup\a^{-3}(S^{-1}(h_{22}))]\bowtie\a^{-2}(h_{21})\cdot_\sigma g\\
=&p[(\a^{-3}(h_1)\rightharpoonup \a^{*2}(q))\leftharpoonup\a^{-3}(S^{-1}(h_{22}))]\bowtie\sigma(\a^{-2}(h_{211}),g_1)\a^{-3}(h_{212})\a^{-1}(g_2)\\
=&p[(\a^{-6}(h_{21})S^{-1}(\a^{-7}(h_{122}))\a^{-5}(h_{11})\rightharpoonup \a^{*2}(q))\leftharpoonup\a^{-4}(S^{-1}(h_{222}))]\\ &\bowtie\sigma(\a^{-2}(h_{121}),g_1)\a^{-3}(h_{221})\a^{-1}(g_2)\\
=&p[(\a^{-4}(h_{21})\rightharpoonup(S^{-1}(\a^{-7}(h_{122}))\a^{-6}(h_{11}))\rightharpoonup \a^{*3}(q))\leftharpoonup\a^{-4}(S^{-1}(h_{222}))]\\ &\bowtie\sigma(\a^{-2}(h_{121}),g_1)\a^{-3}(h_{221})\a^{-1}(g_2)\\
=&q_2(S^{-1}(\a^{-2}(h_{122}))\a^{-1}(h_{11}))p[(\a^{-4}(h_{21})\rightharpoonup \a^{*3}(q_1))\leftharpoonup\a^{-4}(S^{-1}(h_{222}))]\\ &\bowtie\sigma(\a^{-2}(h_{121}),g_1)\a^{-3}(h_{221})\a^{-1}(g_2)\\
=&p_1(1)q_2(S^{-1}(\a^{-2}(h_{122}))\a^{-1}(h_{11}))\sigma(\a^{-2}(h_{121}),g_1)\\
&\a^*(p_2)[(\a^{-4}(h_{21})\rightharpoonup \a^{*3}(q_1))\leftharpoonup\a^{-4}(S^{-1}(h_{222}))]\bowtie\a^{-3}(h_{221})\a^{-1}(g_2)\\
=&p_1(1)q_2(S^{-1}(\a^{-2}(h_{122}))\a^{-1}(h_{11}))\sigma(\a^{-2}(h_{121}),g_1)\\
&\a^*(p_2)[(\a^{-4}(h_{21})\rightharpoonup \a^{*3}(q_1))\leftharpoonup\a^{-4}(S^{-1}(h_{222}))]\bowtie\a^{-3}(h_{221})\a^{-1}(g_2)\\
=&\bar{\sigma}(p_2\o h_1,q_2\o g_2)(\a^*(p_1)\o \a^{-1}(h_2))(\a^*(q_1)\o \a^{-1}(g_2))\\
=&(p\o h)\cdot_{\bar{\sigma}}(q\o g).
\end{aligned}$$
Thus $H^*\bowtie H(\sigma)$ and $D(H)(\bar{\sigma})$ have the same multiplication as well as the $D(H)$-bicomodule structure. For the uniqueness of $\bar{\sigma}$, applying $1\o \varepsilon$ to the above equation.

The proof is completed.
\end{proof}

\begin{lemma}
Let $(H,\a)$ be a Hom-Hopf algebra, $(B,\b)$ a left $H$-module algebra (with action $h\cdot b$) and $(A,\g)$ a left $H$-comodule algebra (with coaction $a\mapsto a_{(-1)}\o a_{(0)}$). Define multiplication on $B\o A$ by
$$(b\o a)(b'\o a')=b(\a^{-2}(a_{(-1)})\cdot \b^{-1}(b'))\o \g^{-1}(a_{(0)})a',$$
for all $b,b'\in B$ and $a,a'\in A$. Then $(B\o A,\b\o\g)$ is a Hom-algebra, which is denoted by $B\ltimes A$.
\end{lemma}

\begin{proof}
The proof is straightforward.

\end{proof}

\begin{proposition}
Let $(H,\a)$ be a Hom-Hopf algebra, $(B,\b)$ a left $H$-module algebra and $(A,\g)$ a left $H$-comodule algebra. Suppose that $(H,B)$ is an admissible pair, then $B\ltimes A$ becomes a left $B\t A$-comodule algebra, with coaction
$$\bar{\r}:B\ltimes A\rightarrow (B\t H)\o(B\ltimes A),~~b\ltimes a\mapsto (b_1\times\a^{-2}(b_{2(-1)})\a^{-1}(a_{(-1)}))\o (\b^{-1}(b_{2(0)})\ltimes a_{(0)}),$$
for all $a\in A,b\in B$.
\end{proposition}

\begin{proof}
Firstly for all $a\in A,b\in B$,
$$\begin{aligned}
&((\b\o\a)\o\bar{\r})\bar{\r}(b\ltimes a)\\
=&(\b(b_1)\times\a^{-1}(b_{2(-1)})a_{(-1)})\o(\b^{-1}(b_{2(0)1})\t \a^{-3}((b_{2(0)2(-1)}))\a^{-1}(a_{(0)(-1)}))\\
&\o(\b^{-2}(b_{2(0)2})\o a_{(0)(0)})\\
=&(\b(b_1)\times\a^{-3}(b_{21(-1)}b_{22(-1)})a_{(-1)})\o(\b^{-1}(b_{21(0)})\t \a^{-3}(b_{22(0)(-1)})\a^{-1}(a_{(0)(-1)}))\\
&\o(\b^{-2}(b_{22(0)})\o a_{(0)(0)})\\
=&(\b(b_{1})\t \a^{-2}(b_{21(-1)})(\a^{-3}(b_{22(-1)})\a^{-1}(a_{(-1)})))\\
&\o(\b^{-1}(b_{21(0)})\t \a^{-3}(b_{22(0)(-1)})\a^{-1}(a_{(0)(-1)}))\o(\b^{-2}(b_{22(0)(0)})\ltimes a_{(0)(0)})\\
=&(b_{11}\t \a^{-2}(b_{12(-1)})(\a^{-3}(b_{2(-1)1})\a^{-2}(a_{(-1)1})))\o(\b^{-1}(b_{12(0)})\t \a^{-2}(b_{2(-1)2})\a^{-1}(a_{(-1)2}))\\
&\o(b_{2(0)}\ltimes \g(a_{(0)}))\\
=&(\Delta\o(\b\o\g))\bar{\r}(b\ltimes a),
\end{aligned}$$
and the counit is easy to check. Then for all $a,a'\in A$ and $b,b'\in B$,
$$\begin{aligned}
&\bar{\r}((b\ltimes a)(b'\ltimes a'))\\
=&\bar{\r}(b(\a^{-2}(a_{(-1)})\cdot \b^{-1}(b'))\ltimes\g^{-1}(a_{(0)})a')\\
=&(b(\a^{-2}(a_{(-1)}))\cdot\b^{-1}(b'))_1\\
&\t \a^{-2}((b(\a^{-2}(a_{(-1)}))\cdot\b^{-1}(b'))_{2(-1)})(\a^{-2}(a_{(0)(-1)})\a^{-1}(a'_{-1}))\\
&\b^{-1}((b(\a^{-2}(a_{(-1)}))\cdot\b^{-1}(b'))_{2(0)})\ltimes\g^{-1}(a_{(0)(0)})a'_{(0)}\\
=&b_1(\a^{-2}(b_{2(-1)})\cdot(\a^{-4}(a_{(-1)11})\cdot\b^{-2}(b'_1)))\\
&\t \a^{-2}(\a^{-1}(b_{2(0)(-1)})(\a^{-3}(a_{(-1)12})\cdot \b^{-1}(b'_2))_{(-1)})\\
&(\a^{-2}(a_{(-1)2})\a^{-1}(a'_{(-1)}))\o\b^{-2}(b_{2(0)(0)})\b^{-1}((\a^{-3}(a_{(-1)12})\cdot \b^{-1}(b'_2))_{(0)})\ltimes a_{(0)}a'_{(0)}\\
=&b_1(\a^{-2}(b_{2(-1)})\cdot(\a^{-4}(a_{(-1)11})\cdot\b^{-2}(b'_1)))\\ &\t[\a^{-3}(b_{2(0)(-1)})\a^{-3}((\a^{-3}(a_{(-1)12})\cdot \b^{-1}(b'_2))_{(-1)}a_{(-1)2})]a'_{(-1)}\\
&\o\b^{-2}(b_{2(0)(0)})\b^{-1}((\a^{-3}(a_{(-1)12})\cdot \b^{-1}(b'_2))_{(0)})\ltimes a_{(0)}a'_{(0)}\\
=&b_1(\a^{-2}(b_{2(-1)})\cdot(\a^{-3}(a_{(-1)1})\cdot\b^{-2}(b'_1)))\t \a^{-3}[b_{2(0)(-1)}(\a^{-1}(a_{(-1)21}) b'_{2(-1)})]a'_{(-1)}\\
&\o\b^{-2}(b_{2(0)(0)})(\a^{-3}(a_{(-1)22})\cdot \b^{-2}(b'_{2(0)}))\ltimes a_{(0)}a'_{(0)}\\
=&b_1((\a^{-4}(b_{2(-1)1})\a^{-3}(a_{(-1)1}))\cdot\b^{-1}(b'_1)))\\
&\t[\a^{-3}(b_{2(-1)2})\a^{-3}(a_{(-1)21})] [\a^{-2}(b'_{2(-1)})\a^{-1}(a'_{(-1)})]\\
&\o\b^{-1}(b_{2(0)})(\a^{-3}(a_{(-1)22})\cdot \b^{-2}(b'_{2(0)}))\ltimes a_{(0)}a'_{(0)}\\
=&b_1(\a^{-4}(b_{2(-1)1})\a^{-3}(a_{(-1)1}))\cdot\b^{-1}(b'_1)))\\
&\t[\a^{-3}(b_{2(-1)2})\a^{-2}(a_{(-1)2})] [\a^{-2}(b'_{2(-1)})\a^{-1}(a'_{(-1)})]\\
&\o\b^{-1}(b_{2(0)})(\a^{-2}(a_{(0)(-1)})\cdot \b^{-2}(b'_{2(0)}))\ltimes \g^{-1}(a_{(0)(0)})a'_{(0)}\\
=&\bar{\r}(b\ltimes a)\bar{\r}(b'\ltimes a').
\end{aligned}$$
The proof is completed.

\end{proof}

\begin{proposition}
Let $(H,\a)$ be a Hom-Hopf algebra, $(B,\b)$ a Hom algebra and Hom coalgebra, and $(H,B)$ an admissible pair. Suppose that $\sigma$ be a normalized and convolution invertible right Hom-2-cocycle on $H$ and consider the left $H$-comodule algebra $H_\sigma$. Then the map $\tilde{\sigma}:(B\t H)\o(B\t H)\rightarrow k$ given by
$$\tilde{\sigma}(b\t h,b'\t h')=\varepsilon_B(b)\varepsilon_B(b')\sigma(h,h'),$$
is a normalized and convolution invertible right Hom-2-cocycle on $B\t H$, and $(B\t H)_{\tilde{\sigma}}=B\ltimes H_\sigma$ as left $B\t H$-comodule algebra. Moreover $\tilde{\sigma}$ is unique with this property.
\end{proposition}

\begin{proof}
For all $b,b'\in B$ and $h,h'\in H$,
$$\begin{aligned}
&(b\t h)\cdot_{\tilde{\sigma}}(b'\t h')\\
=&(\b^{-1}\o\a^{-1})(b_1\t \a^{-2}(b_{2(-1)})\a^{-1}(h_1))(\b^{-1}\o\a^{-1})(b'_1\t \a^{-2}(b'_{2(-1)})\a^{-1}(h'_1))\\
&\varepsilon_B(b_{2(0)})\varepsilon_B(b'_{2(0)})\sigma(h_2,h'_2)\\
=&(b\t \a^{-1}(h_1))(b'\t \a^{-1}(h'_1))\sigma(h_2,h'_2)\\
=&b(\a^{-3}(h_{11})\cdot \b^{-1}(b'))\t \a^{-2}(h_{12})\a^{-1}(h'_1)\sigma(h_2,h'_2)\\
=&b(\a^{-2}(h_1)\cdot\b^{-1}(b'))\ltimes\a^{-2}(h_{21})\a^{-1}(h'_1)\sigma(\a^{-1}(h_{22}),h'_2)\\
=&b(\a^{-2}(h_1)\cdot\b^{-1}(b')\ltimes\a^{-1}(h_1)\cdot_\sigma h'\\
=&(b\ltimes h)(b'\ltimes h'),
\end{aligned}$$
thus the multiplication in $(B\t H)_{\tilde{\sigma}}$ coincides with the one in $B\ltimes H_\sigma$ which is associative, so $\tilde{\sigma}$ is a right 2-cocycle, and we have $(B\t H)_{\tilde{\sigma}}=B\ltimes H_\sigma$ as algebras. Obviously they have the same left $B\t H$-comodule structure, and easy to prove that $\tilde{\sigma}$ is normalized and convolution invertible with the inverse $\tilde{\sigma}^{-1}(b\t h,b'\t h')=\varepsilon_B(b)\varepsilon_B(b')\sigma^{-1}(h,h')$. For the uniqueness of $\tilde{\sigma}$, apply $\varepsilon_B\o\varepsilon_H$ to both sides of
$$(b\t h)\cdot_{\tilde{\sigma}}(b'\t h')=(b\ltimes h)(b'\ltimes h').$$
The proof is completed.

\end{proof}

\begin{definition}
Let $(H,\a)$ be a Hom-Hopf algebra and $(A,\b)$ an $H$-bicomodule algebra with comodule structures $A\rightarrow A\o H, a\mapsto a_{<0>}\o a_{<1>}$ and $A\rightarrow H\o A, a\mapsto a_{[-1]}\o a_{[0]}$. $(M,\mu)$ is called a left-right Yetter-Drinfeld module over $(H,A,H)$ if $M$ is a left $A$-module (action denoted by $\cdot$) and a right $H$-comodule (coaction denoted by $m\mapsto m_{(0)}\o m_{(1)}$) such that
\begin{eqnarray}
\label{B}
\b(a_{<0>})\cdot m_{(0)}\o\a^2(a_{<1>})\a(m_{(1)})=(a_{[0]}\cdot m)_{(0)}\o(a_{[0]}\cdot m)_{(1)}\a^{2}(a_{[-1]}),
\end{eqnarray}
for all $a\in A,h\in H$ and $m\in M$. The category of Yetter-Drinfeld modules is denoted by $_AYD(H)^H.$
\end{definition}

\begin{remark}
Note that when $A=H$ the above definition coincides with the Hom Yetter-Drinfeld module introduced in \cite{MP}
\end{remark}

Let now $\sigma$ be a normalized and convolution invertible lazy Hom-2-cocycle on $H$. Consider the $H$-bimodule algebra $H(\sigma)$ and the associated category $_{H(\sigma)}YD(H)^H$. For an object $(M,\mu)$ in this category, the compatibility (\ref{B}) becomes
\begin{eqnarray}
\label{C}
\a(h_1)\cdot m_{(0)}\o\a^2(h_2)\a(m_{(1)})=(h_2\cdot m)_{(0)}\o(h_2\cdot m)_{(1)}\a^{2}(h_1),
\end{eqnarray}
which is the very compatible condition in the category $_{H}YD(H)^H$. It is easy to see that \ref{C} is equivalent to
\begin{eqnarray}
\label{D}
(h\cdot m)_{(0)}\o(h\cdot m)_{(1)}=\a^{-1}(h_{21})\cdot m_{(0)}\o[\a^{-2}(h_{22})\a^{-1}(m_{(1)})]S^{-1}(h_1).
\end{eqnarray}

\begin{proposition}
Let $\sigma$ be a normalized and convolution invertible lazy Hom-2-cocycle on $H$. Let $(M,\mu)$ be a finite dimensional object in $_{H(\sigma)}YD(H)^H$. Then
\begin{itemize}
  \item [(1)] $(M^*,(\mu^{-1})^*)$ is an object in $_{H(\sigma^{-1})}YD(H)^H$ with the following structures:
\begin{eqnarray}
&&<h\cdot f,m>=<f,S_1(h)\cdot\mu^{-2}(m)>\\
&&f_{(0)}(m)f_{(1)}=f(\mu^{-2}(m_{(0)}))S^{-1}(\a^{-2}(m_{(1)})),
\end{eqnarray}
for all $h\in H,m\in M$ and $f\in M^*$.
  \item [(2)] $(M^*,(\mu^{-1})^*)$ is an object in $_{H(\sigma^{-1})}YD(H)^H$ with the following structures:
\begin{eqnarray}
&&<h\cdot f,m>=<f,S_2(h)\cdot\mu^{-2}(m)>\\
&&f_{(0)}(m)f_{(1)}=f(\mu^{-2}(m_{(0)}))S(\a^{-2}(m_{(1)})),
\end{eqnarray}
for all $h\in H,m\in M$ and $f\in M^*$.
\end{itemize}
\end{proposition}

\begin{proof}
 We only prove (1) while (2) could be proven similarly. First $M^*$ is a right $H$-comodule with the structure (4.20), and since $S_1$ is an algebra anti-isomorphism, $M^*$ is a left $H(\sigma^{-1})$-module. We only need to verify the compatible condition. For all $h\in H, m\in M$ and $f\in M^*$,
$$\begin{aligned}
&(\a(h_1)\cdot f)_{(0)})\a^2(h_2)\a(f_{(1)})\\
=&f_{(0)}(S_1\a(h_1))\cdot\mu^{-2}(m))\a^2(h_2)\a(f_{(1)})\\
=&f(\mu^{-2}(S_1\a(h_1)\cdot\mu^{-2}(m))_{(0)}))\a^2(h_2)S^{-1}(\a^{-1}((S_1\a(h_1)\cdot\mu^{-2}(m))_{(1)}))\\
=&f(\mu^{-2}(\a^{-1}((S_1\a(h_1))_{21})\cdot\mu^{-2}(m_{(0)})))\\
&\a^2(h_2)S^{-1}(\a^{-3}((S_1\a(h_1))_{22})\a^{-4}(m_{(1)}))S^{-1}\a^{-1}(((S_1\a(h_1))_1))\\
\stackrel{(4.15)}{=}&f(S\a^{-2}(h_{112})\cdot\mu^{-4}(m_{(0)}))\a^2(h_2)S^{-1}((S\a^{-2}(h_{111}))\a^{-4}(m_{(1)}))S^{-1}(S_1(h_{12}))\\
=&f(S\a^{-2}(h_{112})\cdot\mu^{-4}(m_{(0)}))\\
&\sigma^{-1}(S(h_{1221}),h_{1222})(\a(h_2)\a^{-1}(h_{121}))[S^{-1}(\a^{-3}(m_{(1)}))\a^{-1}(h_{111})]\\
=&f(S\a^{-1}(h_{12})\cdot\mu^{-4}(m_{(0)}))\sigma^{-1}(S(h_{212}),h_{221})\a^{-1}(h_{222}h_{211})[S^{-1}(\a^{-3}(m_{(1)}))h_{11}]\\
\stackrel{(4.12)}{=}&f(S\a^{-1}(h_{12})\cdot\mu^{-4}(m_{(0)}))\sigma^{-1}(S(h_{21}),h_{22})[S^{-1}(\a^{-2}(m_{(1)}))\a(h_{11})]\\
=&(h_2\cdot f)_{(0)}(m)(h_2\cdot f)_{(1)}\a^2(h_1)
\end{aligned}$$
\end{proof}

\begin{remark}
In the above proposition, when $\sigma$ is trivial, $M\in\  _{H}YD(H)^H$ and these are the usual left and right duals of $M$ in $_{H}YD(H)^H$.
\end{remark}


\begin{thebibliography}{aa}

\bibitem{AGS} L. Alvarez-Gaum谷, C. Gomez, G. Sierra, Quantum group interpretation of some conformal field theories, Phys. Lett. B 220(1每2)(1989): 142--152.
\bibitem{AS} N. Aizawa, H. Sato, q-deformation of the Virasoro algebra with central extension, Phys. Lett. B 256(1) (1991): 185--190.

\bibitem{BC} J. Bichon and G. Carnovale, Lazy cohomology: An analogue of the Schur multiplier for arbitrary Hopf algebras, J. Pure Appl. Algebra, 204(2006): 627--665.

\bibitem{BCM} R. Blattner,M. Cohen, S. Montgomery,. Crosed products and inner actions
of Hopf algebras. Trans. Am. Math. Soc., 298(1986): 671--711.

\bibitem{BM} R. Blattner, S. Montgomery. Crossed products and Galois extensions of Hopf
algebras. Pac. J. Math., 137(1989): 37--54.

\bibitem{DT} Y. Doi, M. Takeuchi. Cleft comodule algebras for a bialgebra. Comm.
Alg., 14(1986): 801--817.

\bibitem{CG} S. Caenepeel and I. Goyvaerts, Monoidal Hom-Hopf algebras, Comm. Alg., 39(2011): 2216--2240.

\bibitem{Ca} G. Carnovale, Some isomorphisms for the Brauer groups of a Hopf algebra, Comm. Alg., 29(2001):5291--5305.

\bibitem{C1} M. Chaichian, D. Ellinas, Z. Popowicz, Quantum conformal algebra with central extension, Phys. Lett. B 248(1每2) (1990): 95--99.

\bibitem{C2} M. Chaichian, A.P. Isaev, J. Lukierski, Z. Popowi, P. Prevnajder, q-deformations of Virasoro algebra and conformal dimensions, Phys. Lett. B 262(1) (1991): 32--38.

\bibitem{C3} M. Chaichian, P. Kulish, J. Lukierski, q-deformed Jacobi identity, q-oscillators and q-deformed infinite-dimensional algebras, Phys. Lett. B 237(3每4) (1990): 401--406.

\bibitem{C4} M. Chaichian, Z. Popowicz, P. Prevnajder, q-Virasoro algebra and its relation to the q-deformed KdV system, Phys. Lett. B 249(1) (1990): 63每-65.

\bibitem{CP} J. Cuadra and F. Panaite, Extending lazy 2-cocycles on Hopf algebras and lifting projective representations afforded by them, J. Algebra, 313(2007):695--723.

\bibitem{Dr} V.G. Drinfel＊d, Quantum groups, Proc. Internat. Congr. Math. Berkeley,  1(1986): 789--820.

\bibitem{HLS} J. T. Hartwig, D. Larsson, S. D. Silvestrov, Deformations of Lie algebras using $\sigma$-derivations, J. Algebra 295(2006): 314--361.

\bibitem{H} N. Hu, $q$-Witt algebras, $q$-Lie algebras, $q$-holomorph structure and representations,
Algebra Colloq. 6 (1999): 51--70.

\bibitem{LS}
L. Liu and B. Shen, Radford's biproducts and Yetter-Drinfel'd modules for monoidal Hom-Hopf algebras, J. Math. Phys. 55(2014), 031701.


\bibitem{LW} D. W. Lu and S. H. Wang, The Drinfel'd Double versus the Heisenberg
 Double for Hom-Hopf Algebras, J. Alg. Appl., 2015

\bibitem{MP2} A. Makhlouf and F. Panaite, Twisting operators, twisted tensor products and smash products for Hom-associative algebras, arXiv: 1402.1893[math.QA].

\bibitem{MP} A. Makhlouf and F. Panaite, Hom-L-R-smash products, Hom-diagonal crossed products and the Drinfel'd double of a Hom-Hopf algebra, J. Algebra, 2015(441):314--343.

\bibitem{MS1} A. Makhlouf and S. D. Silvestrov, Hom-algebra structure, J. Gen. Lie Theory Appl. 2(2008), 52--64.

\bibitem{MS2} A. Makhlouf and S. D. Silvestrov, Hom-algebras and Hom-coalgebras, J. Alg. Appl., 9(2010), 553--589.

\bibitem{R} D. Radford, The structure of Hopf algebras with a projection, J. Algebra, 1985(92):322--347.

\bibitem{WJZ} S. Wang, Z. Jiao and W. Zhao, Hopf algebra structures on crossed products, Comm. Alg., 1998(26):1293--1303.

\bibitem{Y} D. Yau, Hom-bialgebras and comodule Hom-algebras, Int. E. J. Algebra, 8(2010): 45--64.
\end{thebibliography}
\end{document}